\documentclass[a4paper]{article}

\usepackage[english]{babel}
\usepackage[utf8x]{inputenc}
\usepackage[T1]{fontenc}
\usepackage{graphicx}
\usepackage{caption, subcaption}
\usepackage[title]{appendix}
\usepackage{tikz}
\usepackage{pgfplots}
\usepackage{cutwin}
\graphicspath{ {images/} }
\usetikzlibrary{intersections,decorations.pathreplacing,decorations.markings,calc,angles,quotes,arrows.meta,pgfplots.fillbetween,patterns}

\usepackage[a4paper,top=2.8cm,bottom=3cm,left=2.7cm,right=2.7cm,marginparwidth=1.75cm]{geometry}

\usepackage{amsmath,yhmath,amsfonts}
\usepackage{amsthm}
\usepackage{amssymb}
\usepackage[hidelinks]{hyperref}
\usepackage{cleveref}
\usepackage{epstopdf}
\usepackage{comment}
\usepackage{todonotes}
\usepackage{float}
\usepackage[sort&compress,numbers]{natbib}
\newtheorem{thm}{Theorem}
\newtheorem{lemma}[thm]{Lemma}

\theoremstyle{definition}

\newtheorem{rmk}[thm]{Remark}
\numberwithin{equation}{section}
\numberwithin{thm}{section}

\newcommand{\C}{\mathcal{C}}
\newcommand{\D}{\partial D}
\renewcommand{\S}{\mathcal{S}}
\newcommand{\R}{\mathbb{R}}

\renewcommand{\i}{\mathrm{i}}
\newcommand{\de}{\,\mathrm{d}}
\newcommand{\ddp}[2]{\frac{\partial#1}{\partial#2}}
\renewcommand{\v}{\mathbf{v}}
\renewcommand{\u}{\mathbf{u}}

\renewcommand*{\Im}{\operatorname{Im}}

\makeatletter
\newcommand{\neutralize}[1]{\expandafter\let\csname c@#1\endcsname\count@}
\makeatother

\title{Landscape of wave focusing and localisation at low frequencies}
\author{ Bryn Davies\thanks{\footnotesize Department of Mathematics, Imperial College London, London, UK (email: bryn.davies@imperial.ac.uk).} \and Yiqi Lou\footnotemark[1]}
\date{}

\begin{document}
\maketitle
	
	\begin{abstract}
		High-contrast scattering problems are special among classical wave systems as they allow for strong wave focusing and localisation at low frequencies. We use an asymptotic framework to develop a landscape theory for high-contrast systems that resonate in a subwavelength regime. Our from-first-principles asymptotic analysis yields a characterisation in terms of the generalised capacitance matrix, giving a discrete approximation of the three-dimensional scattering problem. We develop landscape theory for the generalised capacitance matrix and use it to predict the positions of three-dimensional wave focusing and localisation in random and non-periodic systems of subwavelength resonators.
	\end{abstract}
	\vspace{0.5cm}
	\noindent{\textbf{Mathematics Subject Classification (MSC2010):} 35J05, 35C20, 35P20, 74A40, 78A48.
		
	\vspace{0.2cm}
	
	\vspace{0.5cm}
	
\section{Introduction}

Wave propagation in random and complex media is often challenging to understand and can lead to a range of exotic behaviour. One interesting phenomenon that can occur in these systems is the tendency for waves to be focused and localised by the material. In extreme cases, this can amount to the complete absence of diffusion (known as \emph{Anderson localisation}) \cite{anderson1958absence}. This behaviour has been observed in a variety of systems, including electromagnetic \cite{laurent2007localized, schwartz2007transport, segev2013anderson} and acoustic \cite{condat1987observability, weaver1990anderson} waves, as well as quantum systems \cite{billy2008direct}. In spite of the extensive body of theory that was developed to describe this phenomenon during the second half of the twentieth century \cite{lagendijk2009fifty}, a notable question remained: given a random material, is it possible to predict \emph{where} wave energy will be localised? Landscape functions were developed to solve this problem \cite{filoche2012universal}.

The theory of landscape functions provides an upper bound on the eigenmodes of a differential problem \cite{filoche2012universal}. When this bound is tight, the eigenmodes are forced to inherit the shape of the landscape function. Hence, if the eigenmode in question is either exponentially localised or strongly focused (in the sense that it has a significant local maximum), then the landscape function can be used to predict the possible locations where it could be localised. This idea has been used to predict localisation in a range of continuous \cite{filoche2012universal} and discrete \cite{lyra2015dual} random systems. The link between localised modes and the landscape function has subsequently been made more precise, with \cite{arnold2019localization} showing how the landscape can be used to predict the support of the localised mode.

Crucially, the constant in the upper bound given by the landscape function scales in proportion to the eigenvalue of the mode in question \cite{filoche2012universal}. As a result, when it comes to predicting the specific regions of localisation in a random material, the method works well for the Schr\"odinger equation because the interesting localised modes typically occur at low frequencies. Conversely, in classical wave systems, the low-frequency modes are typically delocalised and any focusing or localisation occurs at higher frequencies, meaning little useful information can be deduced from the landscape function \cite{colas2022crossover}. An approach for defining a high-frequency landscape was developed for discrete systems by \cite{lyra2015dual}. For continuous classical wave systems, an alternative solution was proposed by \cite{colas2022crossover}, which relies on shifting the eigenvalue of the differential operator and optimising with respect to the energy shift, using the ideas from \cite{arnold2019localization} to define localisation regions. A summary of this approach can be found online at \cite{cottereau2023talk}.

In this work, we study the scattering of scalar waves by arrays of finitely many highly contrasting material inclusions. The very large material contrast means the inclusions act as strong scatterers and exhibit \emph{subwavelength resonance}. That is, resonance in a regime where the incident wavelength is much larger than the dimensions of the scatterers. As a result, these systems are special among classical wave systems as they allow for strong wave focusing and localisation at low frequencies. An example of a strongly focused subwavelength resonant mode is shown in \Cref{fig:modeexample1}, for a system of 15 randomly positioned spherical high-contrast material inclusions. In this case, the system does not support localisation since it has open boundary conditions and energy is allowed to escape to the fair field. Nevertheless, it is clear from \Cref{fig:modeexample1} that the resonant mode's profile, when restricted to a neighbourhood of the scatterers, has its amplitude strongly focused in a single region of three-dimensional space. The aim of this work is to use landscape functions to predict where this focusing will occur. As we will introduce in \Cref{sec:prelim}, the high-contrast regime considered in this work is motivated by the example of \emph{Minnaert resonance}, which is the deeply subwavelength acoustic resonance of an air bubble in water \cite{minnaert1933musical, devaud2008minnaert}.

\begin{figure}
	\centering
	\includegraphics[width=0.75\linewidth,clip,trim=1cm 0 6cm 0]{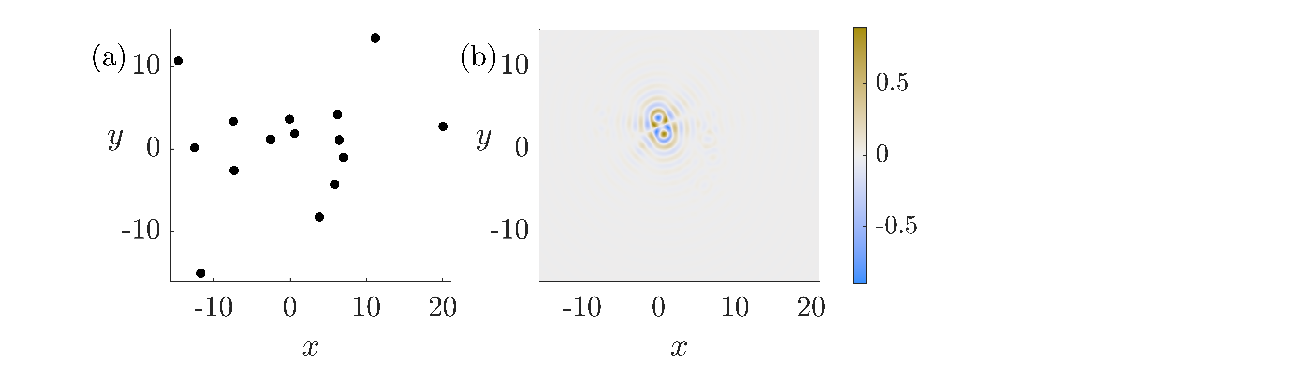}
	\caption{We develop a landscape theory to characterise low-frequency wave scattering by systems of high-contrast subwavelength resonators. (a) A randomly distributed array of 15 spherical subwavelength resonators. (b) One of the subwavelength resonant modes of the system, whose amplitude is focused in a region of three-dimensional space by the resonators.} \label{fig:modeexample1}
\end{figure}

The subwavelength resonance of systems of high-contrast inclusions can be characterised using asymptotic analysis. In particular, it has been shown that the resonant frequencies and associated resonant modes of these systems are given, at leading order, by the eigenvalues and eigenvectors of the \emph{generalised capacitance matrix} \cite{ammari2021functional, ammari2022wave}. This is a generalisation of the capacitance matrix that often appears in electrostatics, having been introduced by Maxwell \cite{maxwell1873treatise} to model the relationship between the distributions of potential and charge in a many-body system of conductors. The application of capacitance coefficients to the high-contrast, low-frequency setting considered here is qualitatively similar, in the sense that the generalised capacitance matrix can be thought of as capturing the strength of the coupling interactions between each of the inclusions. An alternative asymptotic strategy would be to use two-scale homogenisation, which has been applied to the particular problem of localised modes in systems of high-contrast inclusions by \cite{cherdantsev2009spectral, kamotski2018localized}. However, in this work we favour the convenience of the discrete approximation provided by the generalised capacitance matrix.

We will exploit the discrete asymptotic approximation given by the generalised capacitance matrix to develop a landscape theory for predicting the positions of wave localisation in these subwavelength resonant systems. We will see that, using the properties of capacitance coefficients, it is possible to prove a landscape inequality within this discrete framework. This argument is similar to the ideas of \cite{lemut2020localization}. We will see that this inequality still suffers from the same curse described above, whereby little useful information about an eigenmode's profile can be obtained from the landscape inequality if it has a relatively large eigenvalue. We deal with this by using a transformed version of the generalised capacitance matrix to define an ``upper'' landscape that captures the localisation of higher frequency modes. This exploits the ideas developed by \cite{lyra2015dual} and taked advantage of the fact that we have a discrete approximation for the continuous system. Note, however, that although these modes occur at ``higher'' frequencies than the lowest frequency modes, they are still deeply subwavelength. Hence, they can still be thought of as low-frequency modes, they just lie at the upper end of the subwavelength resonant spectrum. 

An important subtlety of the problem considered in this work is that it is three dimensional. Many of the previous applications of landscape function theory have been for one- or two-dimensional differential systems. One simple reason for this is that it is difficult to plot and compare three-dimensional data \cite{cottereau2023talk}. In our high-contrast, low-frequency setting, our asymptotic model reduces the three-dimensional wave scattering problem to a discrete eigenvalue problem, with the matrix dimensions equal to the number of scatterers. Since the three-dimensional behaviour is fully determined, at leading order, by the vector of values on the finite number of resonators, we can understand the resonant modes from this finite set of values. This provides a way to visualise the multi-dimensional data easily and facilitates the comparisons required to apply landscape theory.

Another important subtlety of three-dimensional scattering problems is that it is unclear when, in general, Anderson localisation will occur. For example, \cite{skipetrov2014absence} demonstrated that there is no Anderson localisation when light is scattered by a random three-dimensional ensemble of point scatterers. Comparing this with the varied results of other three-dimensional studies, such as \cite{aegerter2007observation, conti2008dynamic, kondov2011three, skipetrov2019search}, highlights the subtlety and difficulty of understanding Anderson localisation in three dimensions. In this work, we will study systems of finitely many resonators surrounded by free space. As discussed, the radiation conditions imposed on the far field mean that energy is allowed to escape from the system so there is no global localisation. However, we can already see from \Cref{fig:modeexample1} that this system can support strong focusing of modes.

We will begin by recalling the discrete asymptotic framework offered by the generalised capacitance matrix in \Cref{sec:prelim}. We will highlight only the details needed for this study, more comprehensive reviews can be found in \cite{ammari2021functional, ammari2022wave}. Then, in \Cref{sec:landscape} we will present the main theoretical results, which are landscape inequalities tailored to both the upper and lower parts of the subwavelength resonant spectrum. Finally, we will demonstrate this approach on several examples in \Cref{sec:examples}, including both random arrangements of resonators and perturbed periodic arrays.

\section{Setting and preliminaries} \label{sec:prelim}
\subsection{Problem statement}

This study is concerned with the scattering of scalar waves by an array of material inclusions in three dimensions. We will denote these material inclusions as $D_1,D_2,\dots,D_N$, where $N\in\mathbb{N}$ is the total number of inclusions. Since we will study the resonant properties of this system, we will refer to $D_1,\dots,D_N$ as \emph{(subwavelength) resonators} from here on. We assume that the resonators $D_1,\dots,D_N$ are bounded subsets of $\mathbb{R}^3$ with H\"older continuous boundaries. That is, the boundary of each resonator $\D_n$ is such that there is some $0<s<1$ so that $\D_n\in C^{1,s}$ (meaning it is locally the graph of a differentiable function whose derivatives are H\"older continuous with exponent $s$). We will use the notation $D$ to denote the collection of all the resonators, given by the disjoint union of the $N$ resonators
\begin{equation}
D=\bigcup_{n=1}^N D_n.
\end{equation}

We suppose that the interior of each resonator is homogeneous, such that the resonator $D_n$ contains material with constant wave speed $v_n$. Similarly, we suppose that the resonators are surrounded by a homogeneous background material with wave speed $v$. We consider the propagation of time-harmonic scalar waves with frequency $\omega$ and solve a Helmholtz problem for the pressure field $u\in H^1_\mathrm{loc}(\mathbb{R}^3)$, given by

\begin{equation} \label{eq:helmholtz_equation_3d}
\begin{cases}
\left( \Delta + \omega^2 v^{-2} \right) u = 0 & \text{in } \R^3\setminus\overline{D}, \\
\left( \Delta + \omega^2 v_n^{-2} \right) u = 0 & \text{in } D_n, \text{ for } n=1,\dots,N, \\
u_+ - u_- = 0 & \text{on } \D,\\
\delta \ddp{u}{\nu}\big|_+ -  \ddp{u}{\nu}\big|_- = 0 & \text{on } \D, \\
u^s := u - u^\mathrm{in} \text{satisfies an outgoing radiation condition} & \text{as } |x|\to\infty,
\end{cases}
\end{equation}
where $u^\mathrm{in}$ is the incoming wave. Here, $\nu$ is the outward pointing unit normal vector, $\delta>0$ is a non-dimensional parameter and the subscripts $+$ and $-$ denote the limits from outside and inside the boundary respectively. Since $\omega$ is allowed to take complex values, some care is needed to define an appropriate outgoing radiation condition. This condition is needed to specify the behaviour of the solution in the far field, such that energy is only able to radiate outwards and the problem is well posed. When $\Im \omega\geq0$, the condition is the famous Sommerfeld radiation condition, which is given by
\begin{equation} \label{eq:src3}
\lim\limits_{|x|\to\infty} |x| \left(\ddp{}{|x|}-\i \omega v^{-1}\right)u=0.
\end{equation}
In the case that $\Im\omega>0$, \eqref{eq:src3} implies that $u(x)$ decays exponentially as $|x|\to\infty$ \cite[Theorem~3.6]{colton1983integral}. Conversely, when $\Im\omega<0$, the resonant solutions that are of interest in this work are allowed to be exponentially growing for large $|x|$. In this case, one natural approach is to reduce the unbounded problem to a problem posed within a large ball $\Omega\subset\mathbb{R}^3$ containing all the resonators (\emph{i.e.} $\cup_{n=1}^N \overline{D_n}\subset \Omega$). Then, the two versions of the problem can be related using a so-called \emph{capacity operator}, as outlined in \cite[Section~3]{ammari2020perturbation}. With this reformulation, $\omega$ can be identified using the boundary value problem on $\Omega$. Alternatively, viewing \eqref{eq:helmholtz_equation_3d} as an eigenvalue problem of the form $\mathcal{L}u=\omega^2u$, resonant frequencies are poles of the resolvent $(\mathcal{L}-\omega^2)^{-1}$. This is well defined for $\omega^2$ on the real axis, thanks to the Sommerfeld radiation condition, and can be continued meromorphically into the complex plane. See \emph{e.g.} \cite{dyatlov2019mathematical, gopalakrishnan2008asymptotic} for details and examples of this process.

The parameter $\delta$ in \eqref{eq:helmholtz_equation_3d} is a non-dimensional parameter that describes the contrast between the materials inside and outside the resonators. We are interested in the high-contrast limit that is characterised by $\delta$ being very small. This regime is motivated by the phenomenon of \emph{Minnaert resonance}, which is a strongly monopolar resonance of highly contrasting material inclusions. A canonical setting for Minnaert resonance to occur is when the material inclusions are air bubbles surrounded by water \cite{minnaert1933musical, devaud2008minnaert}. In this case, $\delta$ is the ratio of the density of air and water, having a value of approximately $10^{-3}$. 

We will characterise the subwavelength resonance of this system \eqref{eq:helmholtz_equation_3d} as an asymptotic property. Firstly, we define a \emph{resonant mode} to be a non-trivial solution $u$ that exists when the incoming wave $u^\mathrm{in}$ is zero. When such a solution $u$ exists for some frequency $\omega$, we say that $\omega$ is the associated \emph{resonant frequency}. Then, we will define a \emph{subwavelength} resonant mode to be any resonant mode $u$ which is such that its resonant frequency $\omega$ depends continuously on $\delta$ and satisfies
\begin{equation} \label{eq:asymp}
\omega\to0 \quad\text{as}\quad \delta\to0.
\end{equation}
This condition will allow us to perform asymptotic analysis in the high-contrast limit $\delta\to0$ in \Cref{sec:asymptotics}. We will also assume that the wave speeds are approximately constant in the sense that $v=O(1)$, $v_n=O(1)$ and $v/v_n=O(1)$ for all $n=1,\dots,N$ as $\delta\to0$.

\subsection{Boundary integral operators}

We will use boundary integral operators to represent the solutions to \eqref{eq:helmholtz_equation_3d}. This has two main advantages. Firstly, our subsequent analysis will hold for a broad collection of different shapes of resonators (provided they have sufficiently smooth boundaries). Secondly, by taking an appropriate ansatz in terms of boundary integral operators, we can reduce the three-dimensional differential problem \eqref{eq:helmholtz_equation_3d} to a problem posed only on the boundary of the resonators $\D$. 

The key boundary integral operator that we will need is the single layer potential. This is a Green's function operator that uses the standard Helmholtz Green's function, given by
\begin{equation}
G^k(x)=-\frac{\exp(\i k |x|)}{4\pi|x|}.
\end{equation}
Then, the single layer potential is the operator $\S_{D}^k:L^2(\D)\to H^1_\mathrm{loc}(\R^3)$ given by
\begin{equation} \label{defn:singleL3}
	\S_{D}^k[\varphi](x) = \int_{\D} G^k(x-y) \varphi(y)\, \de\sigma(y),\quad x\in\R^3, \varphi\in L^2(\D).
\end{equation}
Various properties of the single layer potential and its use in scattering problems can be found in \cite{ammari2009layer, colton1983integral}. One key property that we will need is that $\S_D^0$ is invertible as a map from $L^2(\D)$ to $H^1(\D)$. The other useful fact is that the single layer potential can be used to represent a resonant mode $u$, which is a solution to \eqref{eq:helmholtz_equation_3d} in the case that $u^\mathrm{in}=0$, as
\begin{equation} \label{eq:representation}
u(x) = \begin{cases}
\S_D^{\omega/v}[\phi](x) & x\in \R^d\setminus\overline{D}, \\[0.3em]
\S_{D}^{\omega/v_n}[\psi](x) & x\in D_n,
\end{cases}
\end{equation}
where $\psi,\phi\in L^2(\D)$ are density functions that need to be found. The value of this representation is that a solution of the form \eqref{eq:representation} necessarily satisfies the Helmholtz equations and the radiation condition in problem \eqref{eq:helmholtz_equation_3d}. Hence, it remains only to find $\psi,\phi\in L^2(\D)$ such that the transmission conditions on the boundary $\D$ in \eqref{eq:helmholtz_equation_3d} are satisfied.

\subsection{Asymptotic methods} \label{sec:asymptotics}

Under the assumption of the asymptotic regime \eqref{eq:asymp}, we can make use of the well-known low-frequency asymptotic expansions of the single layer potential. In particular, we have that 
\begin{equation} \label{eq:Sasymp}
\S_D^{\omega/v}=\S_D^0+O(\omega),
\end{equation}
as $\omega\to0$, with convergence in the sense of the operator norm for bounded linear operators on $L^2(\D)$. We will not recall all the details of the asymptotic analysis here and, instead, refer the interested reader to Section~3.2.3 of \cite{ammari2022wave}. The main idea is that in the limiting case, when $\delta=\omega=0$, we have an operator whose spectrum is straightforward to understand. In particular, \eqref{eq:helmholtz_equation_3d} reduces to Laplace's equation with Neumann boundary conditions on each of the $N$ disjoint domains $D_1,\dots,D_N$. It is clear that this operator has an $N$-dimensional eigenspace (whose eigenfunctions are constant on each of the individual resonators). Then, non-zero $\delta>0$ and $\omega$ can be studied as an asymptotic perturbation of this limiting case. In particular, we can prove the existence of $N$ subwavelength resonant frequencies, see Theorem~3.2.4 of \cite{ammari2022wave}.

\begin{lemma}
A system of $N$ subwavelength resonators supports exactly $N$ subwavelength resonant frequencies with positive real parts, counted with their multiplicities (that is, a single eigenvalue should be counted once, a double eigenvalue should be counted twice, and so on).
\end{lemma}

In the $\delta\to0$ and $\omega\to0$ limit, we have an $N$-dimensional eigenspace spanned by the functions
\begin{equation}
S_j^\omega(x) = \begin{cases}
\S_D^{\omega/v}[\psi_j](x), & x\in\R^3\setminus\overline{D}, \\
\S_D^{\omega/v_n}[\psi_j](x), & x\in D_n, \ n=1,\dots,N,
\end{cases}
\end{equation}
where
\begin{equation} \label{defn:psi}
\psi_j:=(\S_D^0)^{-1}[\chi_{\D_j}]. 
\end{equation}
Here, $\chi_{\D_j}$ is the characteristic function of $\D_j$. We can use these functions to develop an asymptotic characterisation of the subwavelength resonant modes. The key quantity that emerges from the asymptotic expansion \eqref{eq:Sasymp} are \emph{capacitance coefficients}, which are real numbers $C_{ij}$ given by
\begin{equation} \label{cap}
C_{ij}:=-\int_{\D_i} \psi_j \de\sigma,
\end{equation}
for $i,j=1,\dots,N$. These are the standard capacitance coefficients, as introduced by Maxwell \cite{maxwell1873treatise}. To account for the resonators potentially having different sizes or containing different material parameters, we need to introduce the \emph{generalised capacitance coefficients} $\C_{ij}$, given by
\begin{equation} \label{GCM}
\C_{ij} := \frac{v_i^2}{|D_i|} C_{ij},
\end{equation}
where $|D_i|$ is the volume of $D_i$. It turns out that the eigenvalues and eigenvectors of the generalised capacitance matrix determine the subwavelength resonant frequencies and subwavelength resonant modes of the system of resonators, at leading order. This is made precise by the following result, that was proved in Lemma~3.2.3 and Theorem~3.2.5 of \cite{ammari2022wave}. 

\begin{lemma} \label{thm:res}
Let $\lambda^{(1)},\dots,\lambda^{(N)}$ and $\v^{(1)},\dots,\v^{(N)}$ be the eigenvalues and associated eigenvectors of the generalised capacitance matrix $\C$, defined in \eqref{GCM}. Then, as $\delta\to0$, the subwavelength resonant frequencies of a system of $N$ subwavelength resonators satisfy the asymptotic formula
\begin{equation*}
\omega^{(n)}=\sqrt{\delta \lambda^{(n)}} + O(\delta),
\end{equation*}
for each $n=1,\dots,N$. Further, the associated resonant modes  satisfy
\begin{equation}\label{modeexpansion}
u^{(n)}=\sum_{j=1}^N (\v^{(n)})_j \S_D^0[\psi_j](x)+O(\sqrt{\delta}),
\end{equation}
for each $n=1,\dots,N$.
\end{lemma}

\begin{rmk} \label{rmk:identical}
A direct corollary of \Cref{thm:res} is that if the resonators are all identical (having the same size and shape as well as the same internal material parameters), then we only need to study the capacitance matrix $C$, as defined in \eqref{cap}. Since $C$ and $\C$ differ by a common $v_i^2/|D_i|$ factor, their eigenvalues differ by a similar factor. As a result, the expansion for the resonant frequencies takes the form
\begin{equation}
\omega^{(n)}=\sqrt{\delta\lambda^{(n)} \frac{v_1^2}{|D_1|}} + O(\delta),
\end{equation}
for each $n=1,\dots,N$, where $\lambda^{(1)},\dots,\lambda^{(N)}$ are the eigenvalues of the capacitance matrix \eqref{cap}. Crucially, however, in this special case the expansion for the resonant modes $u^{(n)}$ in \Cref{thm:res} is unchanged if $\v^{(1)},\dots,\v^{(N)}$ are the eigenvectors of the capacitance matrix. 
\end{rmk}

We will use the asymptotic expansions presented in this section to study wave localisation within a system of high-contrast resonators. In particular, the aim is to predict where the resonant modes $u_n$ are localised. Thanks to \Cref{thm:res} and \Cref{rmk:identical}, we know that the resonant modes are determined, at leading order in the asymptotic parameter $\delta$, by the eigenvectors of the (generalised) capacitance matrix. In particular, the definition of $\psi_j$ \eqref{defn:psi} means that the functions $\S_D^0[\psi_j](x)$ take the value 1 on the boundary of $D_j$ and vanish on the boundaries of all the other resonators. Hence, \eqref{modeexpansion} tells us that the (leading-order) values of a given resonant mode on each resonator are given by the entries of the corresponding eigenvector of the (generalised) capacitance matrix.

\section{Subwavelength landscape theory} \label{sec:landscape}

We wish to develop a landscape theory for predicting focusing and localisation in the subwavelength regime, using the capacitance matrix approximation presented above. That is, we wish to develop upper bounds for the resonant modes, analogous to the results of \cite{filoche2012universal}. To do so, it is important to understand some properties of the capacitance coefficients $C_{ij}$. Using some basic properties of the single layer potential, which can be found in \emph{e.g.} \cite{ammari2009layer}, we can derive an alternative representation for the capacitance coefficients, as given by the following lemma.
\begin{lemma} \label{lem:altcapdefn}
	For each $i=1,\dots,N$, define the function $V_i:\R^3\to\R$ to be the unique solution to the exterior boundary value problem
	\begin{equation*}
	\left\{
	\begin{array} {ll}
	\Delta V_i  = 0 & \mathrm{in } \ \R^3 \setminus \overline{D}, \\
	V_i=\delta_{ij}  & \mathrm{on }\ \partial D_j, \ \mathrm{ for } \ j=1,\dots,N, \\
	V_i(x)=O\left(|x|^{-1}\right) & \mathrm{as }\ |x|\to\infty,
	\end{array}
	\right.
	\end{equation*}
	where $\delta_{ij}$ is the Kronecker delta. Then, the capacitance coefficients, defined in \eqref{cap}, are given by
	\begin{equation*}
	C_{ij}=\int_{\R^3\setminus D} {\nabla V_i}\cdot \nabla V_j \de x, \quad\text{for } i,j=1,\dots,N.
	\end{equation*}
\end{lemma}

Given the representation from \Cref{lem:altcapdefn}, it is straightforward to prove the following useful properties of the capacitance coefficients, as discussed in \cite{diaz2011positivity}.

\begin{lemma} \label{lem:posdef}
The capacitance coefficients, defined in \eqref{cap}, satisfy $C_{ij}=C_{ji}$ and $C_{ii}>0$ for any $i,j=1,\dots,N$. Further, if $i\neq j$, then $C_{ij}\leq0$ and, finally, we have that 
$$C_{ii}>\sum_{j\neq i} |C_{ij}|,$$
for any $i=1,\dots,N$.
\end{lemma}

The following property follows from the results of \cite{ostrowski1937determinanten}, but we include a simple proof here for completeness. Our argument is inspired by the approach of \cite{lemut2020localization}, who used the language of comparison matrices.

\begin{lemma} \label{lem:Cinv}
The capacitance matrix is invertible and its inverse $C^{-1}$ has non-negative real entries.
\end{lemma}
\begin{proof}
The invertibility of $C$ follows immediately from \Cref{lem:posdef}, since the inequality $C_{ii}>\sum_{j\neq i} |C_{ij}|$ for any $i=1,\dots,N$ implies positive definiteness. Further, since $C$ is positive definite we have that 
\begin{equation} \label{eq:intinv}
C^{-1}=\int_0^{\infty} \exp(-Ct) \de t.
\end{equation}
Now, take some $\mu>0$ such that $\mu>\max_n C_{nn}$ then we can define $L_\mu$ as
\begin{equation} \label{defnL}
L_\mu:= \mu I -C
\end{equation}
and $L_\mu$ will have non-negative entries. As a result, we see that
\begin{equation}
\exp(-Ct)_{ij} = \exp(-\mu t) \exp(L_\mu t)_{ij}= \exp(-\mu t) \sum_{n=0}^\infty \frac{t^n}{n!} \left(L_\mu^n\right)_{ij}\geq0,
\end{equation}
for any $1\leq i,j,\leq N$ and any $t>0$. Hence, \eqref{eq:intinv} is an integral of $N\times N$ non-negative integrands, so each entry of $C^{-1}$ must be non-negative.
\end{proof}

With Lemma~\ref{lem:Cinv} in hand, we are ready to prove the main result of our landscape theory for subwavelength resonant systems.

\begin{thm} \label{thm:landscape}
Define the landscape vector $\u$ as the solution to
\begin{equation*}
C\u=J_{N,1},
\end{equation*}
where $J_{N,1}\in\R^N$ is the vector of ones. Then for any eigenpair $(\lambda,\v)\in(0,\infty)\times\R^N$, which satisfies
\begin{equation*}
C\v=\lambda \v,
\end{equation*}
it holds that
\begin{equation}\label{landineq}
\frac{|\v_i|}{\|\v\|_{\infty}}\leq \lambda \u_i,
\end{equation}
for each $i=1,\dots,N$, where $\|\v\|_{\infty}:=\max_{j=1,\dots,N}|\v_j|$.
\end{thm}
\begin{proof}
From the fact that $C\u=J_{N,1}$ we have for any $i=1,\dots,N$ that
\begin{equation}  \label{eq:proof1}
\sum_{j=1}^N (C^{-1})_{ij} = \u_i.
\end{equation}
Then, we have for any $i=1,\dots,N$ that 
\begin{equation} \label{eq:proof2}
\lambda^{-1} |\v_i| = |(C^{-1}\v)_i| \leq \sum_{j=1}^N |(C^{-1})_{ij} | |\v_j| \leq \|\v\|_{\infty}\sum_{j=1}^N |(C^{-1})_{ij} | = \|\v\|_{\infty}\sum_{j=1}^N (C^{-1})_{ij},
\end{equation}
where the final equality follows from \Cref{lem:Cinv}. Combining \eqref{eq:proof1} and \eqref{eq:proof2} gives the result.
\end{proof}

Ultimately, when it comes to predicting the shapes of resonant modes, the inequality in \Cref{thm:landscape} suffers from the same problem as described in the introduction: that when $\lambda$ is large the inequality \eqref{landineq} doesn't reveal any useful information about the shape of the eigenvector $\v$. This means that resonant modes belonging to the lowest part of the spectrum can be captured effectively but modes at the top of the subwavelength spectrum are likely to be described in less detail. However, we can take inspiration from the matrix $L_\mu$ that was defined in \eqref{defnL} to develop a specific landscape theory for the modes at the top of the subwavelength spectrum. To do so, we will need to introduce the notion of a \emph{comparison matrix}, as used previously by \emph{e.g.} \cite{lemut2020localization, ostrowski1937determinanten}. For an $N\times N$ matrix $A$, the associated comparison matrix is the $N\times N$ matrix $A^\ddag$ given by
\begin{equation} \label{compmat}
(A^\ddag)_{ij}:=\begin{cases} |A_{ii}| & \text{if } i=j, \\ -|A_{ij}| & \text{if }i\neq j.\end{cases}
\end{equation}
Using the comparison matrix of $L_\mu$, for $\mu$ sufficiently large, we will be able to define an ``upper'' landscape vector $\hat\u$ that provides a tighter bound on the eigenvalues at the top of the subwavelength spectrum. This is a variant of the high energy landscape developed by \cite{lyra2015dual} for tridiagonal discrete systems.

\begin{thm} \label{thm:upper}
Let $\mu$ be the positive constant given by
\begin{equation} \label{muchoice}
\mu:=2\max_{i=1,\dots,N}C_{ii}>0.
\end{equation}
Then, we can define the upper landscape vector $\hat\u$ as the solution to
\begin{equation*}
(\mu-C)^\ddag\hat\u=J_{N,1},
\end{equation*}
where $J_{N,1}\in\R^N$ is the vector of ones and the superscript $\ddag$ denotes the comparison matrix, defined in \eqref{compmat}. Then for any eigenpair $(\lambda,\v)\in(0,\infty)\times\R^N$, which satisfies
\begin{equation*}
C\v=\lambda \v,
\end{equation*}
it holds that
\begin{equation} \label{eq:upperland}
\frac{|\v_i|}{\|\v\|_{\infty}}\leq (\mu-\lambda) \hat\u_i,
\end{equation}
for each $i=1,\dots,N$, where $\|\v\|_{\infty}:=\max_{j=1,\dots,N}|\v_j|$.
\end{thm}
\begin{proof}
We will use the notation $L_\mu=\mu I-C$, as was used previously in \eqref{defnL}. The choice of $\mu$ in \eqref{muchoice} means we can use \Cref{lem:posdef} to see that 
\begin{equation} \label{eq:diagdom}
(L_\mu)_{ii}\geq C_{ii}>\sum_{j\neq i} |C_{ij}|=\sum_{j\neq i} |(L_\mu)_{ij}|.
\end{equation}
Hence, we know that $L_\mu$ is positive definite and invertible. We now need to compare $L_\mu$ with its associated comparison matrix $L_\mu^\ddag$. The general result here is known as the \emph{comparison theorem}, which says that for an $N\times N$ matrix $A$, if its comparison matrix $A^\ddag$ is positive definite, then it holds that $|(A^{-1})_{ij}|\leq (A^\ddag)_{ij}$, for all $1\leq i,j\leq N$. This was first proved in \cite{ostrowski1937determinanten} (in German). An alternative proof (in English) can be found in the Supplemental Material of \cite{lemut2020localization}. In our case, $L_\mu^\ddag$ inherits the property of being diagonally dominated from \eqref{eq:diagdom}, so must be positive definite. Thus, we have that 
\begin{equation}
|(L_\mu^{-1})_{ij}|\leq ((L_\mu^\ddag)^{-1})_{ij}.
\end{equation}
Notice also that if an eigenvector satisfies $C\v=\lambda \v$, then we have that $L_\mu\v=(\mu I-C)\v=(\mu-\lambda)\v$, where the choice of $\mu$ as \eqref{muchoice} implies that $\mu-\lambda>0$. As a result, we can show that
\begin{equation}
\frac{1}{\mu-\lambda}|\v_i|=|(L_\mu^{-1}\v)_i| \leq \sum_{j=1}^N |(L_\mu^{-1})_{ij}||\v_i| \leq \|\v\|_\infty\sum_{j=1}^N ((L_\mu^\ddag)^{-1})_{ij} = \|\v\|_\infty \hat\u_i,
\end{equation}
for any $i=1,\dots,N$. 
\end{proof}

The value of the upper landscape developed in \Cref{thm:upper} is that the constant $(\mu-\lambda)$ in the inequality \eqref{eq:upperland} is now such that it is smallest when the eigenvalue $\lambda$ is largest, meaning that this will give the tightest bound on the modes at the top of the subwavelength spectrum. We will demonstrate this with some numerical examples in the following section. This result extends the work of \cite{lyra2015dual} and is qualitatively similar to the approach taken in \cite{colas2022crossover}, where they similarly used an eigenvalue shift to develop an ``optimised'' landscape function (where the optimisation was performed over the energy shift parameter). Their approach relied on using Webster’s transformation \cite{webster1919acoustical} to transform the Helmholtz equation to a Schr\"odinger problem, for which energy levels can be shifted freely.

\section{Examples} \label{sec:examples}

One of the practical challenges that makes it difficult to study wave focusing and localisation in three dimensions is the simple fact that it is hard to visualise three-dimensional functions. This is another challenge that is overcome through the use of a discrete asymptotic approximation in this work (based on boundary integral operators and the capacitance matrix), since the three-dimensional resonant modes $\psi(x)$ are fully determined (at leading order in the asymptotic parameter $\delta$) by the values of the potentials on the $N$ resonators. That is, it is sufficient to study the eigenvectors $\v^{(1)},\dots,\v^{(N)}\in\mathbb{R}^N$ of the capacitance matrix. These vectors can be easily plotted against their index. An example of this is shown in \Cref{fig:modeexample}. \Cref{fig:modeexample}(a) shows a cross section of a resonant mode that appears to be strongly focused and has been plotted in a plane on which the centres of the resonators all lie (by deisgn, see \Cref{sec:random}). \Cref{fig:modeexample}(b) shows the same mode plotted only on the resonators themselves. These leading-order constant values are given by the entries of the corresponding eigenvector of the capacitance matrix. Finally, \Cref{fig:modeexample}(c) shows the same eigenvector with the entries plotted against their index. Although the horizontal axis is non-physical is this case, the plot still contains meaningful information about the extent to which the resonant mode is focused in a region of space (as opposed to its amplitude being distributed across the resonator system).

\begin{figure}
	\centering
	\includegraphics[width=\linewidth]{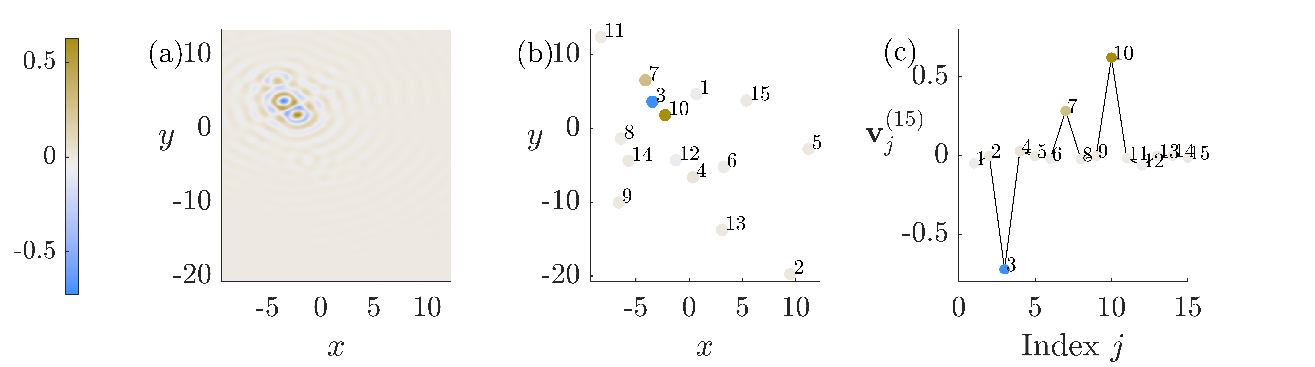}
	\caption{The focusing of a resonant mode can be characterised by the values of the field on each resonator. A system of 15 identical spherical resonators with radius 1 is modelled using the capacitance matrix. The positions are chosen to be normally distributed, with the constraints that the resonators must not overlap and their centres must lie on the $z=0$ plane. (a) The 15\textsuperscript{th} resonant mode $u^{(15)}(x,y,z)$ in the $z=0$ plane. (b) The 15\textsuperscript{th} mode plotted on each of the 15 resonators (with the value given by the eigenvector of the capacitance matrix). (c) The entries of the 15\textsuperscript{th}  eigenvector $\v^{(15)}$ of the capacitance matrix, plotted on a non-physical axis. The colour scale is the same for all three plots.} \label{fig:modeexample}
\end{figure}

The numerical computations are performed using multipole expansions. We consider spherical resonators, in which case it is convenient to express the density functions $\psi_n$ from \eqref{defn:psi} in terms of a basis consisting of spherical harmonics. Calculating the expression of $\psi_n$ within this basis relies on understanding the image of each basis function under the operator $\S_D^0$. This is explained in detail in the appendices of \cite{ammari2020topologically}. Given an expression for $\psi_n$ in terms of this basis, it is straightforward to integrate over the boundaries of the resonators to calculate the capacitance coefficients.

It is important that we consider examples where the resonators are not too close together (and certainly do not touch) as it is known that the field blows up between close-to-touching subwavelength resonators \cite{ammari2020close, lizhao}. This is not a problem for the asymptotic approximation in terms of the generalised capacitance matrix (in fact, the behaviour of the capacitance coefficients of close-to-touching spheres is well known \cite{lekner2011capacitance}), nor does it present issues for the landscape theory. In fact, it is likely to be a mechanism for strongly focused fields as it allows the otherwise monopolar scatterers to exhibit dipolar behaviour (\emph{cf.} the mechanism demonstrated in \cite{marti2021dipolar} for elastic plates). However, close-to-touching resonators pose challenges for numerical computations, as unfeasibly large numbers of basis functions would be needed to handle the problem.

\subsection{Random systems} \label{sec:random}

A first example is shown in \Cref{fig:modeexample}, in which 15 identical spherical resonators are distributed randomly on a plane. Since the resonators are identical, it is sufficient to study the eigenvectors of the capacitance matrix (see \Cref{rmk:identical}). The resonators all have unit radius and their centres that are chosen to lie on the plane $z=0$. The $x$ and $y$ coordinates were drawn from a normal distribution with mean zero and variance 100. The distribution of the resonators is shown in \Cref{fig:modeexample}(b). In this system, the 15\textsuperscript{th} subwavelength resonant mode is strongly focused. \Cref{fig:modeexample}(a) shows this resonant mode in the $z=0$ plane. The colour of the circles in \Cref{fig:modeexample}(b) indicates the value of the mode on each resonator. Finally, \Cref{fig:modeexample}(c) shows the corresponding eigenvector of the capacitance matrix. Although the numbering of the resonators is completely arbitrary, the focusing in a neighbourhood of resonators 3, 7 and 10 is clear (and would similarly be so if a different numbering scheme was chosen).

The eigenvectors of the same system is shown in \Cref{fig:2drandom}, where all 15 eigenvectors of the capacitance matrix are compared to the landsapce and the upper landscape. We can see that the landscape (from \Cref{thm:landscape}) accurately reproduces the profile of the first few eigenvectors but struggles to give any meaningful information beyond this. Conversely, the modes at the top of the subwavelength spectrum can be seen to be strongly focused and their shape is predicted well by the upper landscape (from \Cref{thm:upper}).

\begin{figure}
	\centering
	\includegraphics[width=\linewidth]{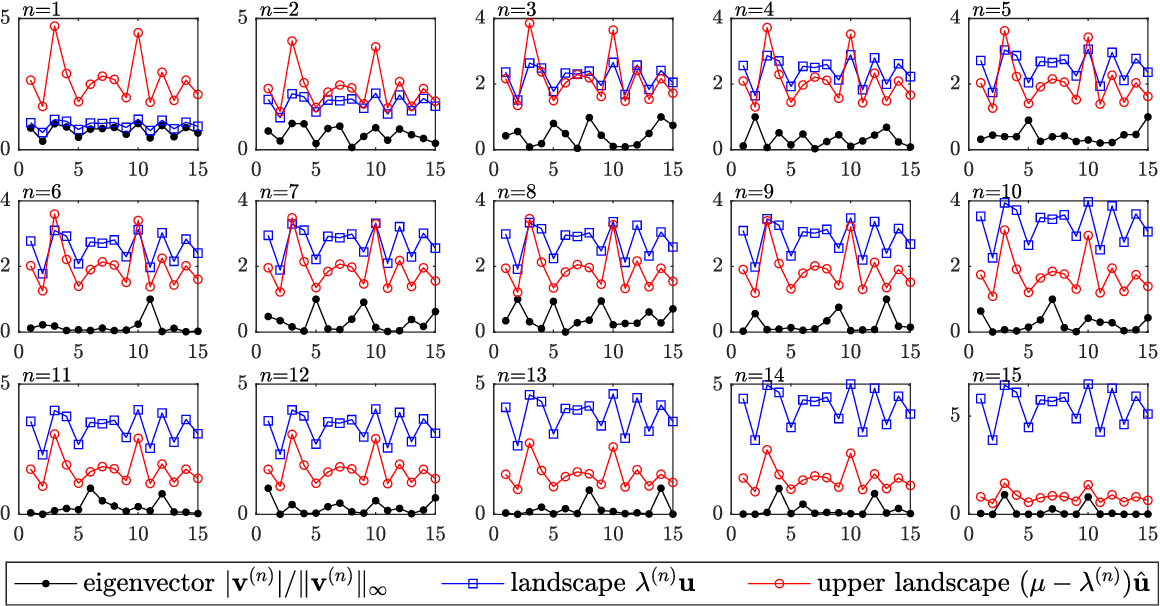}
	\caption{The shapes of the eigenvectors of the capacitance matrix can be predicted by the landscape and the upper landscape. The 15 eigenvectors $\v^{(1)},\dots,\v^{(N)}$ corresponding to a random arrangement of 15 subwavelength resonators are shown, in the same arrangement as was depicted in \Cref{fig:modeexample}. In each plot, the absolute value of the eigenvector is shown, normalised to have maximum equal to 1. The upper bounds given by the landscape $\lambda^{(n)}\u$ and the upper landscape $(\mu-\lambda^{(n)})\hat\u$ are also shown in each case.} \label{fig:2drandom}
\end{figure}

We can also use this approach to study examples where the resonators are free to take random positions in three-dimensional space. For example, let us consider a system of 20 identical spherical resonators (with unit radius) that are distributed uniformly within a $20\times20\times20$ cube. In this case, since the resonators are not constrained to a plane, as in the previous example, we elect not to plot the full resonant mode $u^{(n)}$ and directly opt to plot the eigenvectors of the capacitance matrix on a non-physical axis. Five of the eigenvectors for a realisation of this system are shown in \Cref{fig:3drandom}. We choose the first two modes, for which the landscape from \Cref{thm:landscape} predicts the shape, as well as the top three modes, for which the upper landscape from \Cref{thm:upper} is more revealing.

\begin{figure}
	\centering
	\includegraphics[width=\linewidth]{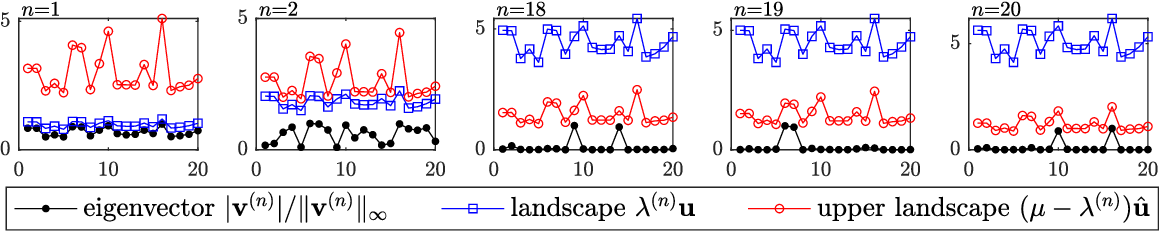}
	\caption{The discrete asymptotic approximation used in this work means the focusing and localisation of resonant modes in three-dimensional systems can be characterised concisely. Here, 20 identical spherical resonators are positioned with centres drawn from a uniform distribution on a $20\times20\times20$ cube. Five of the eigenvectors $\v^{(n)}$ are chosen to be plotted, with each being shown alongside both the landscape $\lambda^{(n)}\u$ and the upper landscape $(\mu-\lambda^{(n)})\hat\u$.} \label{fig:3drandom}
\end{figure}

\begin{rmk}
It is well known that the localisation of eigenvectors of banded random matrices is related to their bandwidth. In particular, eigenvectors are typically localised when the bandwidth is less than $N^{1/2-\epsilon}$ and delocalised when it is greater than $N^{1/2+\epsilon}$ \cite{casati1990scaling}. With this in mind, it is interesting to note that while the capacitance matrix is not banded, its entries decay approximately in proportion to $|i-j|^{-1}$. In particular, \cite[Lemma~4.3]{ammari2020topologically} shows that $C_{ij}$ is approximately inversely proportional to the distance between the centres of resonators $i$ and $j$. While this means that the capacitance matrix is somewhat reminiscent of a banded matrix, this decay is not fast enough for it to behave like a banded matrix with small bandwidth. This is one explanation for why some of the eigenvectors shown in Figures~\ref{fig:2drandom} and~\ref{fig:3drandom} are delocalised. Furthermore, it has been observed that truncating the capacitance matrix to give a banded matrix alters the fundamental physics of the model and gives a poor approximation of the physical system (see \emph{e.g.} \cite[Section~4.2.2.]{ammari2020topologically}).
\end{rmk}

\subsection{Periodic systems with defects} \label{sec:periodic}

Typically, periodic systems (including periodic arrangements of subwavelength resonators) do not give rise to any interesting wave focusing or localisation effects \cite{brillouin1953wave}. Instead, they typically support purely propagating modes (often known as \emph{Bloch} modes) with any insulating properties arising due to spectral gaps. However, if perturbations are made to a periodic material, then this will often lead to the creation of resonant modes that are localised to region of the defect or interface that has been created. This is a popular strategy for designing waveguides and related problems have been studied extensively. This includes systems of subwavelength resonators; several different examples of this are surveyed in \cite{ammari2021functional} and the references therein.

A simple example of a way to perturb a periodic system to create a localised resonant mode is to perturb the material parameters on one of the resonators. This will cause one of the subwavelength resonant modes to be localised in a region of the perturbed resonator. In this case, we will need to use the generalised capacitance matrix to capture the subwavelength behaviour of the system. This setting was studied in detail in \cite{ammari2022anderson}, where a formula was derived for the eigenvalue of the generalised capacitance matrix corresponding to the localised eigenvector. It shows that the localised eigenvector is perturbed away from the spectral band of propagating Bloch modes. This localised mode will exist for any positive perturbation of the material parameter, with the exponential localisation becoming increasingly strong as the perturbation increases.

An example of this system is shown in \Cref{fig:pointdefect} where we have studied identically sized spherical resonators (with unit radius) distributed on a square lattice. Within the central resonator, the wave speed $v_i$ is chosen to be double that in the other resonators. This perturbation is chosen such to be sufficiently large that the mode decays observably within the length of the finite sized array depicted here. In this case, the localised mode appears at the top of the subwavelength spectrum. Hence, we \emph{a priori} expect the upper landscape to be most useful for predicting its shape. The eigenvector is shown as it is distributed in space over the resonators in \Cref{fig:pointdefect}(a) and it is shown alongside both the landscape and the upper landscape in  \Cref{fig:pointdefect}(b). Clearly, the landscape is not useful, so a zoomed plot showing just the mode and the upper landscape can be found in \Cref{fig:pointdefect}(c). The upper landscape unambiguously predicts the localisation on the central resonator.

\begin{figure}
	\centering
	\includegraphics[width=0.9\linewidth]{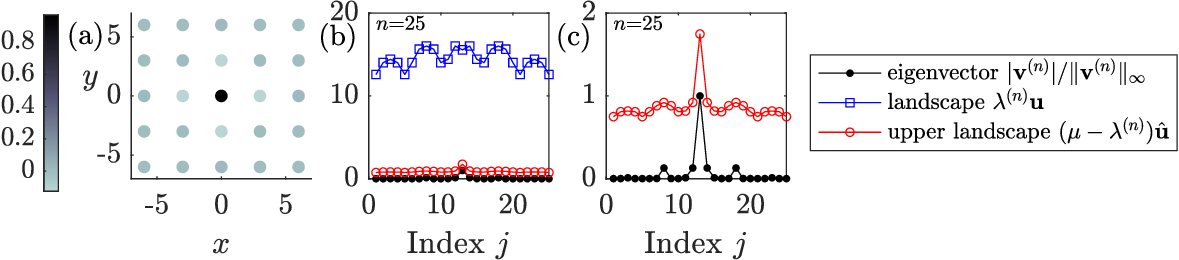}
	\caption{A localised resonant mode can be created by introducing a defect to a periodic arrangement of resonators. 25 spherical resonators are arranged in a square lattice and the wave speed on the central resonator is perturbed to be double that on the other resonators. (a) The value of the mode on each of the 25 resonators. (b)~The profile of the mode compared to both the landscape $\lambda^{(n)}\u$ and the upper landscape $(\mu-\lambda^{(n)})\hat\u$. (c)~The mode compared to just the  upper landscape $(\mu-\lambda^{(n)})\hat\u$, which gives a more insightful prediction of the mode's profile (as would be expected, for this high-frequency subwavelength resonant mode).} \label{fig:pointdefect}
\end{figure}

\section{Concluding remarks}

We have shown that landscape theory can be applied directly to the generalised capacitance matrix, which is used here as a discrete asymptotic model for the subwavelength resonance of high-contrast heterogeneous systems. Using the properties of the capacitance matrix, we were able to obtain simple direct proofs of the landscape inequality (\Cref{thm:landscape}). We also used a transformed version of the capacitance matrix to obtain an ``upper'' landscape, that can describe the top of the subwavelength spectrum (\Cref{thm:upper}). We demonstrated this approach on several examples, including both random arrangements of resonators and perturbed periodic arrays. The discrete asymptotic approximation at the centre of this work means that three-dimensional focusing and localisation problems could be handled and visualised concisely.

The value of landscape theory is particularly apparent when dealing with large systems of many resonators. In this case, substantial computational expense would be required to compute all the eigenvectors of the capacitance matrix and check where they are localised or focused (if at all). Conversely, computing the landscape and upper landscape requires just two matrix equations to be solved and gives direct insight into where energy is likely to be focused within the system.

The two versions of the landscape inequality presented in this work mean that the profile of resonant modes occurring at either end of the subwavelength spectrum can be captured. However, this theory is still less revealing for resonant modes whose resonant frequency falls in the centre of the subwavelength spectrum, especially for very large systems of many resonators. There are several notable examples for which the important localised modes appear in the middle of the subwavelength spectrum, such as the topologically protected modes that are widely used in waveguide design \cite{ammari2020topologically}. Predicting the localisation of waves in quasicrystalline wave systems is a similarly challenging problem. While landscape functions have been applied to quasicrystalline Hamiltonian systems \cite{rontgen2019local}, the localised modes typically appear in the centre of the spectrum (within one of the fractal collection of spectral gaps \cite{ni2019observation}), so suffer from the issue of landscape bounds not being tight enough to reveal useful information.

While understanding wave propagation in random and perturbed periodic media is interesting in its own right and has significant applications (such as to imaging, see \emph{e.g.} \cite{borcea2002imaging}, as well as waveguide design), it also has deep implications for the field of metamaterials and wave guides. When any micro-structured wave control device is manufactured, it will inevitably have some random errors and imperfections introduced. While small random perturbations have minimal effect on subwavelength devices \cite{davies2022robustness}, when these imperfections are sufficiently large it is possible that they will hinder the device's function. It is not yet clear how landscape theory can be used to understand this breakdown in function and, by extension, be used to design devices that are robust with respect to imperfections. However, this is a promising direction for future work.

\section*{Funding}

The work of BD was funded by the Engineering and Physical Sciences Research Council through a fellowship with grant number EP/X027422/1.

\section*{Data availability}

The code used to produce the numerical examples presented in this work is available for download at \url{https://doi.org/10.5281/zenodo.8134312}. No other datasets were generated during this study.

\section*{Conflicts of interest}

The authors declare that they have no conflicts of interest.

\bibliographystyle{abbrv}
\bibliography{references}{}
\end{document}